\documentclass[11pt,a4paper]{amsart}
\usepackage{amsmath,amssymb,amscd,amsthm,amsxtra, mathrsfs, array,soul,ulem}

\usepackage[latin9]{inputenc}

\usepackage[yyyymmdd,hhmmss]{datetime}
\usepackage[usenames,dvipsnames]{xcolor}
\usepackage[
colorlinks=true,
linkcolor=blue,
citecolor=blue,
urlcolor=blue,
pagebackref,
]{hyperref}

\usepackage{ulem}

\newtheorem{theorem}{Theorem}[section]
\newtheorem{lemma}[theorem]{Lemma}
\newtheorem{proposition}[theorem]{Proposition}

\newtheorem{conjecture}[theorem]{Conjecture}

\theoremstyle{definition}
\newtheorem{definition}[theorem]{Definition}

\theoremstyle{remark}
\newtheorem{remark}[theorem]{Remark}
\numberwithin{equation}{section}
\DeclareMathOperator*{\esssup}{ess\,sup}

\newcommand{\absval}[1]{\mbox{$|#1|$}}

\newcommand{\R}{\mathcal{R}}

\newcommand{\ccR }{\mathfrak{R}}

\newcommand{\I }{\mathcal I}

\def\Xint#1{\mathchoice
  {\XXint\displaystyle\textstyle{#1}}%
  {\XXint\textstyle\scriptstyle{#1}}%
  {\XXint\scriptstyle\scriptscriptstyle{#1}}%
  {\XXint\scriptscriptstyle\scriptscriptstyle{#1}}%
  \!\int}
\def\XXint#1#2#3{{\setbox0=\hbox{$#1{#2#3}{\int}$}
    \vcenter{\hbox{$#2#3$}}\kern-.5\wd0}}

\def\avgint{\Xint-}
\newcommand{\vertiii}[1]{{\left\vert\kern-0.25ex\left\vert\kern-0.25ex\left\vert #1 
    \right\vert\kern-0.25ex\right\vert\kern-0.25ex\right\vert}}

\numberwithin{equation}{section}

\definecolor{caro}{rgb}{0.0, 0.5, 0.0}

\begin{document}

 \title[ Biparametric Poincar\'e through maximal]{Some non-standard biparametric Poincar\'e type inequalities through harmonic analysis}

\author[M.E. Cejas]{Mar\'ia Eugenia Cejas}
\address[Mar\'ia Eugenia Cejas]{Departamento de Matem\'atica, Facultad de Ciencias Exactas, Universidad Nacional de La Plata, Calle 50 y 115, (1900) La Plata, Prov. de Buenos Aires, Argentina.}
 \email{ecejas@mate.unlp.edu.ar}

\author[C. Mosquera]{Carolina Mosquera}
\address[Carolina Mosquera]{Department of Mathematics,
Facultad de Ciencias Exactas y Naturales,
University of Buenos Aires, Ciudad Universitaria
Pabell\'on I, Buenos Aires 1428 Capital Federal Argentina} \email{mosquera@dm.uba.ar}

\author[C. P\'erez]{Carlos P\'erez}
\address[Carlos P\'erez]{ Department of Mathematics, University of the Basque Country, IKERBASQUE 
(Basque Foundation for Science) and
BCAM \textendash  Basque Center for Applied Mathematics, Bilbao, Spain}
\email{cperez@bcamath.org}

\author[E. Rela]{Ezequiel Rela}
\address[Ezequiel Rela]{Department of Mathematics,
Facultad de Ciencias Exactas y Naturales,
University of Buenos Aires, Ciudad Universitaria
Pabell\'on I, Buenos Aires 1428 Capital Federal Argentina} \email{erela@dm.uba.ar}

\thanks{M.E.C. is partially supported by grant PICT-2018-03017 (ANPCYT)}
\thanks{C.M. is partially supported by grants UBACyT 20020170100430BA, PICT 2018--03399 and PICT 2018--04027.}
\thanks{C. P. and E.R. have been supported by the Basque Government through the IT1247-19 project, the BERC 2018-
2021 program  and by the Spanish State Research Agency through BCAM Severo
Ochoa excellence accreditation SEV-2017-0718. C.P is also supported  by the  project 
PID2020-113156GB-I00/AEI /10.13039/501100011033 funded by Spanish State Research Agency and acronym ``HAPDE".}
\thanks{E.R. is partially supported by grants PICT-2019--03968, PICT 2018-3399, UBACyT 20020190200230BA and PIP (CONICET) 11220110101018 }


\begin{abstract}
We show some non-standard Poincar\'e type estimates  in the biparametric setting with appropriate weights. We will derive these results using variants from classical estimates exploiting the interplay between maximal functions and fractional integrals. We also provide a sharper result by using extrapolation techniques.
\end{abstract}

\maketitle

\normalem

\begin{center}
\end{center}

\

\section{Introduction and Main Results}\label{sec:Intro-Main}

In the present article we will study non-standard variations of the very well known Poincar\'e-Sobolev inequalities. More precisely, we are interested in biparametric extensions of the classical Poincar\'e-Sobolev inequality of the form
\begin{equation}\label{eq:PSpp*}
\left( \frac{1}{|Q|}\int_{Q} \absval{f-f_{Q}}^{p^{*}}
\right)^{1/p^{*}}
\le
c\,\ell(Q)
\left( \frac{ 1 }{|Q| } \int_{Q} \absval{\nabla f}^{p}\right)^{1/p},
\end{equation}
where $Q$ is any cube in $\mathbb{R}^n$
(from now on we will only consider cubes with sides parallell to the coordinate axes) or euclidean ball, $p$ is a parameter $p>1$ and $p^*$ is the so called Sobolev exponent given by the condition $\frac1p-\frac1n=\frac1{p^{*}}$, defined for $p<n$.

The classical way of proving Poincaré-Sobolev inequalities  is by exploiting the interplay between averaged oscillations and fractional integral operators. Recall that $\I_{1}$ is the standard notation for  the fractional
integral or Riesz potential of order $1$ defined by
\begin{equation}\label{eq:def-fractional}
\I_1(f)(x)=\int_{\mathbb{R}^{n} }\frac{f(y)}{|x-y|^{n-1}}\ dy.
\end{equation}

The key estimate is the following pointwise estimate, allowing the control of the oscillation by the fractional integral

\begin{equation}\label{eq:pointwise1}
|f(x)-f_{Q}| \le c_n\,\I_{1}( |\nabla f| \chi_{ Q } )(x).
\end{equation}
It is an interesting fact that the previous estimate is equivalent to the following averaged result, 

\begin{equation} \label{eq:PI1-1}
\frac{1}{|Q|}\int_{Q} \absval{f-f_{Q}}\le c_n\,
\frac{ \ell(Q) }{|Q| } \int_{Q} \absval{\nabla f},
\end{equation}
as shown first in \cite{FLW} (see an extension in \cite{PR-Poincare}, Theorem 11.3).  Then, estimate \eqref{eq:PSpp*} will follow from the appropriate $L^p\to L^q$ boundedness properties of the classical fractional integral operators $I_{\alpha}$.
We remit to \cite{AH} for an extensive overview of this material. 
The key ideas go back to Sobolev about 100 years ago. 

In the late 90's, in \cite{FPW98} a new method to derive these estimates which avoids the use of potential operators was introduced. This method was considered very recently by the authors in \cite{CMPR-JD}, using ideas from \cite{PR-Poincare}, in the context of rectangles, or more generally in the context of muliparameter analysis. We include below Theorem \ref{thm:AutomejorastrongccR} as a  sample of this kind of results which do not use any pointwise estimate as \eqref{eq:pointwise1} requiring the use of fractional operators. This result provides an inequality holding for the family $\ccR$, defined by rectangles of the form $R=I_1\times I_2$ where $I_1\subset \mathbb{R}^{n_1}$ and $I_2 \subset \mathbb{R}^{n_2}$ are cubes with sides parallel to the coordinate axes and $n:=n_1+n_2$. 
\begin{theorem}\label{thm:AutomejorastrongccR}  \, 
Let $w\in A_{1,\ccR}$ in $\mathbb{R}^{n}$ and let $p \geq 1.$
Let also 
$$
\frac1p-\frac1{p^*}= \frac1{n}\frac1{(1+\log [w]_{A_{1,\ccR}} )}.
$$
Then, there exists a constant $c=c(n,p)>0$ such that for every Lipschitz function $f$ and any $R=I_1\times I_2\in \ccR$, 
\begin{equation}\label{eq:PSccR}
\left\|f-f_R\right \|_{L^{p^*}(R, \frac{wdx}{w(R)})}
\leq c\,
[w]_{A_{p,\ccR}}^{\frac{1}{p}} \left (a_1(R)+a_2(R)
 \right ),
\end{equation}
where 
\begin{equation*}
a_1(R)= \ell(I_1) 
\left\|\nabla_1 f\right \|_{L^{p}(R, \frac{wdx}{w(R)})}
\quad \text{ and }\quad 
a_2(R)= \ell(I_2) 
\left\|\nabla_2 f\right \|_{L^{p}(R, \frac{wdx}{w(R)})}.
\end{equation*}
\end{theorem}

We are using here the following notation: for a given function  $f:U\to\mathbb{R}$ defined on the open set $U\subset \mathbb{R}^{n_1} \times \mathbb{R}^{n_2}$, we will write $f(x)=f(x_1,x_2)$ where $x_1$ stands for the first $n_1$ variables and $x_2$ stands for the remaining $n_2$ variables. 
$\nabla_1f$ will denote the partial gradient of $f$ containing the $x_1$-derivatives and 
similarly  $\nabla_2f$ will denote the partial gradient of $f$ containing the $x_2$-derivatives.

To the extent of our knowledge, these type of results were proved for the first time in \cite{ShiTor}, but this theorem is a  particular case of many others results that can be found in \cite{CMPR-JD}. The latter approach provides much more precise inequalities.

The present  work is an outgrowth of \cite{CMPR-JD}. Indeed, it is quite surprising that a result like \eqref{eq:PSccR} cannot be derived using a version of the pointwise estimate \eqref{eq:pointwise1} based on a multiparameter potential operator version of  $\I_1$ from \eqref{eq:def-fractional}. There is though a multiparametric  counterpart of the fractional integral operator introduced in \cite{CUMP-2004} which leads to a special pointwise inequality and hence to a non-standard Poincaré inequality \eqref{eq:pointwise-oscillation-fractional} and 
\eqref{eq:(1,1)-Poincare-Product}. The  main point of this paper is to improve the $(1,1)$ non-standard Poincaré 
inequality \eqref{eq:(1,1)-Poincare-Product} to the $(p,p)$ case. We stress the fact that unfortunately we cannot use the method in \cite{CMPR-JD} since 
it is not clear how to handle the non-standard oscillation $\pi_R(f)$ appearing in \eqref{eq:pointwise-oscillation-fractional}.

\subsection{Biparameter  Poincar\'e inequalities} \label{subsec:MixedGradient}
We include in this section the main contributions of this article, concerning results on  $(p,p)$-Poincar\'e type inequalities obtained by means of classical arguments involving the interplay between fractional integrals and maximal functions. This approach is known for the case of cubes, but requires some extra work to adapt it to the bi-parametric geometry of rectangles given by products of cubes.
Recall that $I_1$ will always denote a cube in $\mathbb{R}^{n_1}$ while $I_2$ will be a cube in $\mathbb{R}^{n_2}$. Let $\delta_1=\ell(I_1)$ and $\delta_2=\ell(I_2).$

Let us consider the multiparametric fractional operator 
\begin{equation}\label{multiparametric fractional operator}
Tf(x,y)=\int_{\mathbb{R}^{n_1}}\int_{\mathbb{R}^{n_2}} \frac{f(\bar{x},\bar{y})}{|x-\bar{x}|^{n_1-1}|y-\bar{y}|^{n_2-1}}\,d\bar{x}\,d\bar{y}.
\end{equation}
Consider the following notations:
\begin{equation*}
f_{I_1}^y=\frac{1}{|I_1|}\int_{I_1} f(x,y)\,dx
\qquad , \qquad 
f_{I_2}^x=\frac{1}{|I_2|}\int_{I_2} f(x,y)\,dy
\end{equation*}
and 
\begin{equation*}
f_{I_1\times I_2}= \frac{1}{|I_1||I_2|}\int_{I_1\times I_2}f(x,y)\,dydx.
\end{equation*}
We denote $\I_1^{(1)}$ the $n_1-$dimensional fractional integral of order 1 and  $\I_1^{(2)}$ the $n_2-$dimensional fractional integral of order 1. That is, for a function $f:\mathbb{R}^{n_1}\times \mathbb{R}^{n_2}\to \mathbb{R}$, we have that
\begin{equation}\label{eq:def-fractional-1}
\I_1^{(1)}(f)(x,y)=\int_{\mathbb{R}^{n_1} }\frac{f^y(\bar{x})}{|x-\bar{x}|^{n_1-1}}\ d\bar{x} 
\end{equation}
and 
\begin{equation}\label{eq:def-fractional-2}
\I_1^{(2)}(f)(x,y)=\int_{\mathbb{R}^{n_2}}\frac{f^x(\bar{y})}{|y-\bar{y}|^{n_2-1}}\ d\bar{y},
\end{equation}
where $f^x(y)=f(x,y)$ denotes the slice of the function $f$ for a fixed $x\in I_1$ (similarly for $f^y$). 
Then $T$ can be expressed as the composition 
$$T=\I_1^{(1)}\circ \I_1^{(2)}.$$

The main relevance for us to consider this multilinear operator is due 
to its connection with a special non-standard  oscillation introduced, as far as we know, in  
\cite{CUMP-2004} with respect to non-constant average defined by the quantity
\begin{equation*}
 \pi_{I_1\times I_2}(f):=f_{I_1}^y+f_{I_2}^x-f_{I_1\times I_2}
 \end{equation*}
which will play the role of the average over the cube as in the one parameter case. 
To be more precise we have in the following lemma a substitute of the model example \eqref{eq:pointwise1} in the  setting of product spaces where 
the usual gradient is replaced by the following mixed derivatives matrix
$\displaystyle f \mapsto \nabla_x \nabla_y f,$
where 
\begin{equation*}
\displaystyle\nabla_x \nabla_y f=\left( \frac{\partial^2
f}{\partial x_i\, \partial y_j}\right)_{i,j}
 \quad \text{ and} \quad 
\displaystyle|\nabla_x \nabla_y f| = \bigg(\sum_{i,j} \Big|
\frac{\partial^2 f}{\partial x_i\, \partial y_j} \Big|^2
\bigg)^\frac12. 
\end{equation*}

We don't know where this object was first introduced but it was considered by J. M. Wilson in \cite{Wilson-1991-Indiana,Wilson-1995-Rocky} when studying spectral type properties of the two-parameters Schrödinger operators $\Delta_{n_1}\circ \Delta_{n_2}-V$, where $V\in L_{loc}^1(\mathbb{R}^{n_1} \times \mathbb{R}^{n_2})$ and where $\Delta_{n_i}$ is the Laplace operator in $\mathbb{R}^{n_i}$ where the multiparameter Harmonic Analysis theory played a central role like in the present work.

We will rely on the following pointwise inequality, analogous to \eqref{eq:pointwise1}

\begin{lemma}\cite[Proposition 6.1]{CUMP-2004} 
\label{lem:pointwise-oscillation-fractional} Let  $R\in \ccR$ of the form $R=I_1\times I_2,$ as before. Then we have the following pointwise estimate for any   $f\in C^2(R)$
\begin{equation}\label{eq:pointwise-oscillation-fractional} 
|f(x,y)-\pi_{R}(f)(x,y)| \leq T(|\nabla_x \nabla_y f \chi_{R}|)(x,y),
\end{equation}
for every $(x,y)\in I_1\times I_2$.
\end{lemma}

 As a consequence of this we have 
\begin{equation} \label{eq:(1,1)-Poincare-Product}
\int_{R} |f-\pi_{R}(f)|\,dx\,dy \lesssim \ell(I_1) \ell(I_2) 
	\int_{R} |\nabla_x \nabla_y f|\,dx\,dy,
\end{equation}
although there is a more stronger and hence more useful version in \eqref{eq:(1,1)-Poicare-ProductA1}. A natural question arising here is whether  these two estimates are equivalent or not, as it happens in the classical setting of cubes mentioned before.

Here and in the reminder of this article we will use the following notation for (weighted) local $L^p$ norms over a rectangle $R$:
\begin{equation*}
\left\|f\right \|_{L^p\left (R,w\right )}:=
 \left (\int_R |f |^p w dx\right )^{\frac{1}{p}}.
\end{equation*}

Similarly, we will use the standard notation for the weak $(r,\infty)$ (quasi)-norm: for any $0<r<\infty$, measurable $R$ and weight $w$, we define
\begin{equation*}
\| f \|_{L^{r,\infty}\big(R,w\big)}:= \sup_{t>0}  t \, \left(w(\{x\in R:  |f(x)|>t\}) \right )^{1/r}.
\end{equation*}
As usual, when dealing with the constant weight $w\equiv 1$, we will simply omit to mention it.

In our context,  the natural scenario for the study of weighted inequalities is the Lebesgue space $L^p(w)$, where $w$ is  a weight in the class $A_{p,\ccR}$ associated to the family $\ccR$.  Let us first introduce its obvious definition adapted to the geometry of such basis of rectangles. For  a weight $w$ in $\mathbb{R}^{n_1}\times\mathbb{R}^{n_2}$, $n=n_1+n_2$, we will say that $w \in A_{p,\ccR}$ if
\begin{equation}\label{eq:Ap-ccR}
[w]_{A_{p, \ccR}}:=\sup_{R \in \ccR} 
\left( \frac{1}{|R|}\int_R w(x)\,dx\right) \left(\frac{1}{|R|}\int_R w(x)^{-\frac{1}{p-1}}\,dx \right)^{p-1}< \infty, 
\end{equation}
and in the case $p=1$, for a finite constant $c$
\begin{equation}\label{eq:A1-ccR}
\frac{1}{|R|}\int_R w(x)\,dx \leq c\, \inf_R w    \qquad R \in \ccR
\end{equation}
and the smallest of the constants $c$  is denoted by \,$[w]_{A_{1, \ccR}}$.

Our first result is about a quantitative weighted Poincar\'e inequality for the $A_{p, \ccR}$ class.

\begin{theorem}\label{thm:debil multiparametrico}  Let $w$ be a weight in $A_{p,\ccR}$ for $p>1$. Define $R=I_1\times I_2$. Then the  following local weak type Poincar\'e inequality holds for any   $f\in C^2(R),$

\begin{equation*}
\|f-\pi_{R}(f)\|_{L^{p,\infty}(R,w)}
\lesssim [w]_{A_{p, \ccR}}^{\frac{1}{p}+ \frac{1}{p-1}} \ell(I_1)\ell(I_2)
\,\|\nabla_x \nabla_y f\|_{L^{p}(R,w)}.
%
\end{equation*}
\end{theorem}

Note that we only obtain here the weak norm, since the maximal operator is not bounded in $L^1$. This is also reflected in the exponent in the $A_p$ constant, as a sum of $\frac{1}{p}$ from the weak bound used in the first step and $\frac{1}{p-1}$ from the strong bound that we are forced to use as a last step.

We remark here that in the classical setting of cubes, a sort of ``weak implies strong" argument (often called the ``truncation" method) can be applied to obtain the $(p,p)$-Poincar\'e inequality from the weak estimate. Here, since we are not in position to claim that such an argument works for mixed 
gradients $\nabla_x \nabla_y$, we can only propose the following result as a conjecture

\begin{conjecture}Let $w$ be a weight in $A_{p,\ccR}$ for $p>1$. Then for any rectangle of the form $R=I_1\times I_2$ and any $f\in C^2(R)$, we have
\begin{equation*} 
\|f-\pi_{R}(f)\|_{L^{p}(R,w)} \lesssim [w]_{A_{p, \ccR}}^{\frac{1}{p-1} +\frac1p} \ell(I_1) \ell(I_2) 
\,\|\nabla_x \nabla_y f\|_{L^{p}(R,w)}. \\
\end{equation*}
\end{conjecture}

 Another natural question is to find a Sobolev-Poincaré type inequality like \eqref{eq:PSpp*}. Naturally, $T$, as defined in \eqref{multiparametric fractional operator},  should play the role of $\I_1$ in the one-parameter case. However we failed doing this. Further, the method developed in \cite{CMPR-JD}, which avoids the use of fractional operators, cannot be applied here since the initial Poincaré \eqref{eq:(1,1)-Poincare-Product}  is not, as already mentioned,  the standard one and hence a new theory must be developed.

As we mentioned above the ideas presented in the proof of Theorem \ref{thm:debil multiparametrico} can not be extended to the case $p=1$. We can, however, prove the following result.
\begin{proposition} \label{pro:specialPoin11}
Let $w$ be a weight in $A_{1,\ccR}$. Then the following $(1,1)$ Poincar\'e inequality holds uniformly in $R=I_1 \times I_2\in\ccR$ for $f\in C^2(R)$. 
\begin{equation} \label{eq:(1,1)-Poicare-ProductA1}
\int_{R} |f-\pi_{R}(f)|\,w\,dx\,dy \lesssim \ell(I_1) \ell(I_2) [w]_{A_{1, \ccR}}^{2}
	\int_{R} |\nabla_x \nabla_y f|\,w\,dx\,dy.
\end{equation}
\end{proposition}

As a consequence, the following theorem shows that the $(p,p)$-Poincar\'e inequality can be obtained directly but with larger bounds.

\begin{theorem}\label{thm:fuerte multiparametrico - p>1} 
Let $w$ be a weight in $A_{p,\ccR}$ for $p>1$.
Then, for any $R= I_1 \times I_2$ in $\ccR$ any $f\in C^2(R)$ we have that 
\begin{equation*}
\|f-\pi_{R}(f)\|_{L^{p}(R,w)} \lesssim [w]_{A_{p, \ccR}}^{\frac{2}{p-1} } \ell(I_1) \ell(I_2) 
\,\|\nabla_x \nabla_y f\|_{L^{p}(R,w)}. 
\end{equation*}
\end{theorem}

Observe that the behaviour of the exponent as $p\to1$ is not the expected one in view of  \eqref{eq:(1,1)-Poicare-ProductA1}.  We remedy the situation by including here a variation of both Theorem \ref{thm:debil multiparametrico} and Theorem \ref{thm:fuerte multiparametrico - p>1} by means of an  extrapolation type argument.

\begin{theorem}\label{thm: strongPI} 
Let $w$ be a weight in $A_{p,\ccR}$ for $p>1$.
Then, for any $R= I_1 \times I_2$ in $\ccR$ and any $f\in C^2(R)$ we have that 
\begin{equation}\label{eq:stronPI}
\|f-\pi_{R}(f)\|_{L^{p}(R,w)} \lesssim [w]_{A_{p, \ccR}}^{\min\{4, \frac{2}{p-1} \} } \ell(I_1) \ell(I_2) 
\,\|\nabla_x \nabla_y f\|_{L^{p}(R,w)}.
\end{equation}
\end{theorem}

This theorem will be a consequence of a more precise result from Theorem \ref{thm: strong-extrapolation} below involving a ``dual maximal operator'' presented in Section \ref{sec:extrapol}.

\subsection{Fractional Integrals} 
\label{Definicionesprevias}

We include here the needed definitions about fractional integrals and $A_p$ classes of weights.
For $0<\alpha<n$, the fractional integral operator or Riesz
potential  $\I_\alpha$ is defined by 
\begin{equation*}
\I_\alpha f(x)=\int_{R^n} \frac{f(y)}{|x-y|^{n-\alpha}}dy.  
\end{equation*}
A key ingredient in our proofs is the  pointwise estimate in Lemma \ref{lem:max y frac} involving the fractional integral and the Hardy-Littlewood maximal function $M$ defined by
\begin{equation*}\label{eq:M}
Mf(x) = \sup_{Q\ni x } \avgint_{Q} |f(y)|\ dy, 
\end{equation*}
where the supremum is taken over all cubes $Q\subset \mathbb{R}^n$ with sides parallel to the coordinate axes containing the point $x$.

\begin{lemma}\label{lem:max y frac} For every cube $Q\subset \mathbb{R}^n$ we have
\begin{equation*}
\I_1(g\chi_Q)(x)\leq C_n \ell(Q)M(\chi_Q g)(x), \qquad x\in Q.
\end{equation*}
\end{lemma}

In addition, we need to recall the following very well know result about the sharp weighted bound for the Hardy-Littlewood maximal function. To that end, we recall here the definition of Muckenhoupt weights $A_p$ for cubes in $\mathbb{R}^n$.

\begin{definition}
For a given $p\in (1,\infty)$, the Muckenhoupt $A_{p}$ of weights is defined by the condition
\begin{equation}\label{eq:Ap}
[w]_{A_p}:=\sup_{Q}\left(\frac{1}{|Q|}\int_{Q}w(y)\ dy \right)\left(\frac{1}{|Q|}\int_{Q}w(y)^{1-p'}\ dy \right)^{p-1}<\infty, 
\end{equation}
where the supremum is taken over all the cubes $Q$ in $\mathbb{R}^n$.
The limiting case of \eqref{eq:Ap} when $p=1$, defines the class $A_1$; that is, the set of weights $w$ such that
\begin{equation*}
[w]_{ A_1}:=\sup_{Q}\bigg(\avgint_Q w\,dx \bigg) \esssup_{Q} (w^{-1})<+\infty.
\end{equation*}
This is equivalent to $w$ having the property
\begin{equation*}
 Mw(x)\le [w]_{A_1}w(x)\qquad \text{ a.e. } x \in \mathbb{R}^n.
\end{equation*}
\end{definition}

We need to consider, in our bi-parametric setting, the $A_p$ property for the slices $w^x$ and $w^y$. We remark here that
according to \cite[Lemma  6.2]{GCRdF} we have that $w^x \in A_p(\mathbb{R}^{n_2})$ 
and moreover $[w^x]_{A_p(\mathbb{R}^{n_2} )}\leq [w]_{A_{p, \ccR}}$. The same holds for $w^y=w(x,y)$ with $ y \in \mathbb{R}^{n_2}$.

We recall now Buckley's result on sharp weighted $L^p$ norms.
\begin{theorem}\label{thm:buckley} \cite{Buckley} If $1<p<\infty$ and $w\in A_p$ then
	\begin{equation}
	\|Mf\|_{L^{p,\infty}(w)}\lesssim [w]_{A_p}^{1/p}\|f\|_{L^p(w)}
	\end{equation}
	\begin{equation}\|Mf\|_{L^p(w)}\lesssim [w]^\frac{1}{p-1}_{A_p}\|f\|_{L^p(w)}
	\end{equation}
	where in each case the exponents are best possible.
\end{theorem}

\section{Proofs for  Biparameter  Poincaré inequalities} \label{sec:MixedGradient}

We include here the proofs of the main results.

\begin{proof}[Proof of Theorem \ref{thm:debil multiparametrico}]

According to inequality \eqref{eq:pointwise-oscillation-fractional} and using  Lemma \ref{lem:max y frac} twice on each direction $n_i$, we have
\begin{eqnarray*}
\|f-\pi_{R}(f)\|_{L^{p,\infty}(R,w)}  & \leq & c_n
\| T(|\nabla_x \nabla_y f \chi_{R}|) \|_{L^{p,\infty}(R,w)}\\
& \le & c_{n}\ell(I_1)\ell(I_2) 
\| M^{n_2} \circ M^{n_1}(|\nabla_x \nabla_y f \chi_{R}|) \|_{L^{p,\infty}(R,w)}
\end{eqnarray*}
where we denote $M^{n_1}$ and $M^{n_2}$ the $n_1$-dimensional and $n_2$-dimensional Hardy-Littlewood maximal operator respectively. 


Now, for a fixed $\lambda>0$, we denote 
\begin{equation*}
\Omega_{\lambda}:=\{(x,y) \in R: M^{n_2} \circ M^{n_1}(|\nabla_x \nabla_y f \chi_{R}|)(x,y) > \lambda  \}.
\end{equation*}
Also, for each fixed $x \in I_1$ we put
\begin{equation*}
\Omega_{\lambda}^x:=\{ (x,y) : y \in I_2, M^{n_2} \circ M^{n_1}(|\nabla_x \nabla_yf \chi_{R}|)(x,y) > \lambda  \}.
\end{equation*}
Then
\begin{equation*}
w(\Omega_{\lambda})=\int_{\Omega_{\lambda}}w(x,y)\,dx\,dy=\int_{I_1}w^x(\Omega_{\lambda}^x)\,dx.
\end{equation*}
Here, we use again the standard notation of $w^x(y)=w(x,y)$ to denote the slice of the function $w$ for a fixed $x\in I_1$. Now we apply Theorem \ref{thm:buckley} as follows. The first step is to use the sharp weak type  bound for $M^{n_2}$ with the weight $w^x(y)$. Then, 
\begin{align*}
w(\Omega_{\lambda}) &\leq c_{n_2}\frac{[w]_{A_{p, \ccR}}^p}{\lambda^p} \int_{I_1} M^{n_1}(|\nabla_x \nabla_y f \chi_{R}|)(x,y)^p\,dxdy\\
&\leq c_{n,p}\frac{ [w]_{A_{p, \ccR}}^p [w]_{A_{p, \ccR}}^{p'}}{\lambda^p} 
\int_{I_1} |\nabla_x \nabla_y f \chi_{R}(x,y)|^p\,dxdy
\end{align*}
by applying the sharp $(p,p)$ strong bound for $M^{n_1}$ considering now the weight $w^y(x) \in A_p (\mathbb{R}^{n_1})$. This concludes the proof immediately.
\end{proof}

\begin{remark} Note that, as we mentioned before, this argument can not be applied to the case $p=1$, given that the Hardy-Littlewood operator is not bounded in $L^1$.
\end{remark}

We include now the proof of the strong bound for $p>1$ stated in Theorem \ref{thm:fuerte multiparametrico - p>1}.

\begin{proof}[Proof of Theorem \ref{thm:fuerte multiparametrico - p>1}]
Let's write $I=	\int_{R} |f-\pi_{R}(f)|^pw\,dx\,dy $ and recall that $w^x$ is a function on $y$. Then
\begin{eqnarray*}
I &\lesssim & \ell(I_1)^p\ell(I_2)^p \int_{R}[ M^{n_1}  M^{n_2}(|\nabla_x \nabla_y  f \chi_{R}|)(x,y)]^p w(x,y)\, dxdy\\ 
	&\approx &\ell(I_1)^p\ell(I_2)^p\int_{I_1}\int_{I_2}M^{n_1}(M^{n_2}(|\nabla_x \nabla_y f \chi_{R}|))(x,y)^pw^x(y)\,dy \,dx \\ 
&\lesssim &\ell(I_1)^p\ell(I_2)^p[w^x]^{\frac{p}{p-1}}_{A_p(\mathbb{R}^{n_1})}\int_{I_1}\int_{I_2}[M^{n_2}(|\nabla_x \nabla_y f \chi_{R}|)(x,y)]^p w^y(x)\,dxdy\\ 
&\lesssim  & \ell(I_1)^p\ell(I_2)^p [w]_{A_{p, \ccR}}^{\frac{p}{p-1}}[w^y]^{\frac{p}{p-1}}_{A_p(\mathbb{R}^{n_2})} \int_{I_2}\int_{I_1} |\nabla_x \nabla_y f(x,y)|^pw(x,y)\,dy\,dx \\
&\approx &\ell(I_1)^p\ell(I_2)^p [w]_{A_{p, \ccR}}^{\frac{2p}{p-1}} \int_{I_2}\int_{I_1} |\nabla_x \nabla_y f(x,y)|^pw(x,y)\,dy \,dx
\end{eqnarray*}
where in first step we use estimate (\ref{eq:pointwise-oscillation-fractional}) and Lemma \ref{lem:max y frac}. Finally, we apply the sharp bound of the Hardy-Litllewood operator twice, first in $L^p_{w^x}(\mathbb{R}^{n_2})$ and then in $L_{w^y}^p(\mathbb{R}^{n_1})$.
\end{proof}

Now we present the proof of the natural substitute valid for $p=1$.

\begin{proof}[Proof of Proposition \ref{pro:specialPoin11}]
Let's call $I$ to the LHS on \eqref{eq:(1,1)-Poicare-ProductA1}. Then, by Lemma \ref{lem:pointwise-oscillation-fractional} and recalling that 
\begin{equation*}
Tf(x,y)=\int_{\mathbb{R}^{n_1}}\int_{\mathbb{R}^{n_1}} \frac{f(\bar{x},\bar{y})}{|x-\bar{x}|^{n_1-1}|y-\bar{y}|^{n_2-1}}\,d\bar{x}\,d\bar{y}.
\end{equation*}
\begin{eqnarray*}
I &\lesssim & \int_R T(|\nabla_x \nabla_y f \chi_{R}|)w\, dx dy\\
&\lesssim &\int_{R} \int_{R} \frac{|\nabla_x \nabla_y f(\bar{x},\bar{y})|}{|x-\bar{x}|^{n-1}|y-\bar{y}|^{m-1}}\, d\bar{x} d\bar{y}\,w dx dy\\
&\lesssim & \int_{R} |\nabla_x \nabla_y f(\bar{x},\bar{y})| \left[\int_{I_1} \frac{1}{|x-\bar{x}|^{n-1}}  \int_{I_2} \frac{w^x(y)}{|y-\bar{y}|^{m-1}} \, dy\,dx\right] \, d\bar{x} d\bar{y}\\
&\lesssim &\ell(I_2) [w]_{A_{1, \ccR}}
	\int_{R} |\nabla_x \nabla_y f(\bar{x},\bar{y})|\int_{I_1}\frac{w^{\bar{y}}(x)}{|x-\bar{x}|^{n-1}}\,dx\,d\bar{x}\,d\bar{y}\\
	& \lesssim & [w]_{A_{1, \ccR}}^2 \ell(I_1)\ell(I_2)\int_{R}|\nabla_x \nabla_y f(\bar{x},\bar{y})|w(\bar{x},\bar{y})\,d\bar{x}\,d\bar{y},
\end{eqnarray*}
where we have used the following inequalities 
\begin{eqnarray*}
\int_{I_2} \frac{w^{x}(y)}{|y-\bar{y}|^{m-1}} \, dy & \lesssim & \ell(I_2)\, [w]_{A_{1, \ccR}}\,w(x,\bar{y}),\\
\int_{I_1} \frac{w^{\bar{y}}(x)}{|x-\bar{x}|^{n-1}} \, dx
& \lesssim &\ell(I_1)\,[w]_{A_{1, \ccR}}\,w(\bar{x},\bar{y}) \,  .
\end{eqnarray*}
\end{proof}

\section{non-standard Poincaré inequalities and Extrapolation}\label{sec:extrapol}

In this section we will prove Theorem \ref{thm: strongPI} using ideas from extrapolation.

Define, as usual, the maximal operator:
\begin{equation}\label{Maximal-ccR}
M_{\ccR}f(x):=\sup_{x\in R \in \ccR} 
\frac{1}{|R|}\int_R |f(y)|\,dy.
\end{equation}
Using the know weighted estimate on each factor, we obtain the estimate for any $p\in(1,\infty)$:
\begin{equation}\label{eq:Buckley}
\|M_{\ccR}\|_{L^p(w)}\lesssim  [w]_{A_{p, \ccR}}^{\frac{2}{p-1} }.
\end{equation}
We also need  the ``dual" operator 
\begin{equation}\label{dualMaximal-ccR}
M_{\ccR}'f := \frac{M_{\ccR}(fw)}{w}.
\end{equation}
Using that $w^{1-p'}\in A_{p',\ccR}$ and also that $M$ is bounded on
$L^{p'}(w^{1-p'})$, we conclude that $M'$ is bounded on $L^{p'}(w)$ and 
\begin{equation}\label{eq:dual-bound}
\|M'_{\ccR}\|_{L^{p'}(w)}\lesssim  [w]_{A_{p, \ccR}}^{2 }.
\end{equation}

We have the following intermediate result:

\begin{theorem}\label{thm: strong-extrapolation} 
Let $w$ be a weight in $A_{p,\ccR}$ for $p>1$.
Then, for any $R= I_1 \times I_2$ in $\ccR$ and any $f\in C^2(R)$ we have that 
\begin{equation}\label{eq:(p,p)-M'}
\|f-\pi_{R}(f)\|_{L^{p}(R,w)}  \lesssim \|M_{\ccR}'\|_{L^{p'}(w) }^{2} \ell(I_1) \ell(I_2) 
\,\|\nabla_x \nabla_y f\|_{L^{p}(R,w)}
\end{equation}
\end{theorem}

\begin{proof}

 Fix $p>1$, and let $w\in A_{p,\ccR}$.

Therefore, we can define the following Rubio de Francia type  iteration algorithm:
\[ \R' h(x) = \sum_{k=0}^\infty \frac{(M_{\ccR}')^kh(x)}{2^k\|M_{\ccR}'\|^k_{L^{p'}(w)}}. \]
where\, $M_{\ccR}' \,$ is the ``dual" operator \eqref{dualMaximal-ccR} and $(M_{\ccR}')^0=I_{d}$.  Then we have that: \\
(1) \quad $h\le \R'(h)$\\
(2) \quad $\|\R'(h)\|_{L^{p'}(w)}\le
2\,\|h\|_{L^{p'}(w)}$\\
(3) \quad  $[\R'(h)\,w]_{A_{1,\ccR}}\leq 2\,   \|M_{\ccR}'\|_{L^{p'}(w) }$
%


By duality there exists a non-negative function $h\in L^{p'}(w)$, $\|h\|_{L^{p'}(w)}=1$, supported in $R$, such that,
%
\begin{equation*} 
\|f-\pi_{R}(f)\|_{L^{p}(R,w)}
=
\int_{R} |f-\pi_{R}(f)|\, hw\,dx
\leq
\int_{R} |f-\pi_{R}(f)|\, \R' h \,w\,dx
\end{equation*}
and since  $\R'(h)\,w \in A_{1,\ccR} $ \,with\, $[\R'(h)\,w]_{A_{1,\ccR}}\leq 2\,   \|M_{\ccR}'\|_{L^{p'}(w) }$ we can apply Proposition \ref{pro:specialPoin11} to obtain
\begin{equation*} 
\|f-\pi_{R}(f)\|_{L^{p}(R,w)}
\leq c\, \ell(I_1) \ell(I_2) \|M_{\ccR}'\|_{L^{p'}(w) }^{2} \,\int_{R} |\nabla_x \nabla_y f| 
\R' h \,w\,dx.
\end{equation*}
Let us focus now in the last integral and apply H\"older inequality and the properties of $\mathcal{R}'$:
\begin{eqnarray*}
\int_{R} |\nabla_x \nabla_y f| 
\R' h \,w\,dx & \le & \left( \int_{R} |\nabla_x \nabla_y f|^p 
\,w\,dx\right)^{\frac1p}
 \,\left(\int_{\mathbb{R}^n} (\R' h)^{p'} \,w\,dx\right)^{\frac1{p'}}\\
 & \le & \left( \int_{R} |\nabla_x \nabla_y f|^p \,w\,dx\right)^{\frac1p}
 \,\left(\int_{R}  h^{p'} \,w\,dx\right)^{\frac1{p'}}\\
 & = & \left( \int_{R} |\nabla_x \nabla_y f|^p \,w\,dx\right)^{\frac1p}.
\end{eqnarray*}

\end{proof}

We are now able to present the proof of Theorem \ref{thm: strongPI} using, on one hand, the estimate from \eqref{eq:(p,p)-M'}  combined with inequality \eqref{eq:dual-bound} to obtain
\begin{equation*}
\|f-\pi_{R}(f)\|_{L^{p}(R,w)} \lesssim [w]_{A_{p, \ccR}}^4\ell(I_1) \ell(I_2) 
\,\|\nabla_x \nabla_y f\|_{L^{p}(R,w)}. 
\end{equation*}
On the other hand, we already had the estimate provided in Theorem \ref{thm:fuerte multiparametrico - p>1}. Combining the two results, we get the desired result stated in \eqref{eq:stronPI}.

\bibliographystyle{amsalpha}


\begin{thebibliography}{GCRdF85}

\bibitem[AH95]{AH}
D.~R. {Adams} and L.~I. {Hedberg}, \emph{{Function spaces and potential
  theory}}, vol. 314, Berlin: Springer-Verlag, 1995.

\bibitem[Buc93]{Buckley}
S.~M. Buckley, \emph{Estimates for operator norms on weighted spaces and
  reverse {J}ensen inequalities}, Trans. Amer. Math. Soc. \textbf{340} (1993),
  no.~1, 253--272.

\bibitem[CMPR]{CMPR-JD}
M.E. Cejas, C.~Mosquera, C.~Perez, and E.~Rela, \emph{Self-improving
  {P}oincar\'e-{S}obolev type functionals in product spaces}, Journal d'Analyse Mathématique,
  To appear.

\bibitem[CUMP04]{CUMP-2004}
D.~Cruz-Uribe, J.~M. Martell, and C.~P\'{e}rez, \emph{Extrapolation from
  {$A_\infty$} weights and applications}, J. Funct. Anal. \textbf{213} (2004),
  no.~2, 412--439. \MR{2078632}

\bibitem[FLW96]{FLW}
B. Franchi, Guozhen Lu, and Richard~L. Wheeden, \emph{A relationship between
  {P}oincar\'e-type inequalities and representation formulas in spaces of
  homogeneous type}, Internat. Math. Res. Notices (1996), no.~1, 1--14.

\bibitem[FPW98]{FPW98}
B. Franchi, C. P\'erez, and R.~L. Wheeden, \emph{Self-improving
  properties of {J}ohn-{N}irenberg and {P}oincar\'e inequalities on spaces of
  homogeneous type}, J. Funct. Anal. \textbf{153} (1998), no.~1, 108--146.

\bibitem[GCRdF85]{GCRdF}
J. Garc{\'{\i}}a-Cuerva and J.~L. Rubio~de Francia, \emph{Weighted
  norm inequalities and related topics}, North-Holland Mathematics Studies,
  vol. 116, North-Holland Publishing Co., Amsterdam, 1985.


\bibitem[PR19]{PR-Poincare}
C. P\'{e}rez and E. Rela, \emph{Degenerate {P}oincar\'{e}-{S}obolev
  inequalities}, Trans. Amer. Math. Soc. \textbf{372} (2019), no.~9,
  6087--6133.

\bibitem[ST93]{ShiTor}
X. L. Shi and A. Torchinsky, \emph{Poincar\'{e} and {S}obolev
  inequalities in product spaces}, Proc. Amer. Math. Soc. \textbf{118} (1993),
  no.~4, 1117--1124.

\bibitem[{Wil}91]{Wilson-1991-Indiana}
J.~M. Wilson, \emph{{Some two-parameter square function inequalities}},
  {Indiana Univ. Math. J.} \textbf{40} (1991), no.~2, 419--442.

\bibitem[{Wil}95]{Wilson-1995-Rocky}
J. M. Wilson,  \emph{{Eigenvalue estimates for degenerate partial differential
  operators}}, {Rocky Mt. J. Math.} \textbf{25} (1995), no.~3, 1171--1187.

\end{thebibliography}

\end{document}